\documentclass[10pt,leqno]{amsart}
\usepackage{amssymb,amsthm}
\usepackage{amsmath}
\usepackage{pictexwd,dcpic}
\usepackage[all]{xy}
\usepackage{hyperref}
\hypersetup{
    bookmarks=true,         
    unicode=false,          
    pdftoolbar=true,        
    pdfmenubar=true,        
    pdffitwindow=false,     
    pdfstartview={FitH},    
    pdftitle={Universal Deformation Rings of String Modules over a Certain Tame Algebra },    
    pdfauthor={Jose A. Velez-Marulanda},     
    pdfsubject={Representation Theory of Associative Algebras},   
    pdfcreator={Jose A. Velez-Marulanda},   
    pdfproducer={Jose A. Velez-Marulanda}, 
    pdfkeywords={Universal deformation rings} {Frobenius algebras} {Stable endomorphism rings}, 
    pdfnewwindow=true,      
    colorlinks=true,       
    linkcolor=blue,          
    citecolor=blue,        
    filecolor=blue,      
    urlcolor=blue           
}

\usepackage{bbm}

\theoremstyle{definition}
 \newtheorem{definition}{Definition}[section]

\theoremstyle{plain}
 \newtheorem{proposition}[definition]{Proposition}

\theoremstyle{plain}
 \newtheorem{theorem}[definition]{Theorem}

\theoremstyle{definition}
 \newtheorem{example}[definition]{Example}

\theoremstyle{plain}

\theoremstyle{plain}

\theoremstyle{remark}
 \newtheorem{remark}[definition]{Remark}
 \newtheorem{claimproof}[definition]{Claim}
 \newtheorem*{proofclaim}{Proof of Claim}
\newcommand{\Ext}{\mathrm{Ext}}
\newcommand{\End}{\mathrm{End}}

\newcommand{\Mat}{\mathrm{Mat}}
\newcommand{\Hom}{\mathrm{Hom}}
\newcommand{\Ca}{\mathcal{C}}
\newcommand{\Fun}{\mathrm{F}}
\newcommand{\Def}{\mathrm{Def}}
\newcommand{\Sets}{\mathrm{Sets}}

\newcommand{\SEnd}{\underline{\End}}
\newcommand{\A}{\Lambda}
\newcommand{\surjection}{\twoheadrightarrow}
\newcommand{\injection}{\hookrightarrow}

\renewcommand{\k}{\Bbbk}

\renewcommand{\1}{\mathbbm{1}}
\newcommand{\invlim}{\varprojlim}
\newcommand{\Ar}{\A_{\bar{r}}}

\title{Universal Deformation Rings of Strings Modules over a Certain Symmetric Tame Algebra}

\thanks{The author was supported by the Release Time for Research Scholarship of the Office of Academic Affairs and the Faculty Research Seed Grant of the Office of Sponsored 
Programs \& Research Administration at the Valdosta State University.}

\author{Jos\'e A. V\'elez-Marulanda}
\address{Department of Mathematics \& Computer Science, Valdosta State University,
2072 Nevins Hall, 1500 N. Patterson St, Valdosta, GA,  31698-0040}
\email{javelezmarulanda@valdosta.edu}

\begin{document}
\maketitle
\begin{abstract}
Let $\k$ be an algebraically closed field, let $\A$ be a finite dimensional $\k$-algebra and let $V$ be a $\A$-module with stable endomorphism ring isomorphic to $\k$. If $\A$ is 
self-injective then $V$ has a universal deformation ring $R(\A,V)$, which is a complete local commutative Noetherian $\k$-algebra with residue field $\k$. Moreover,  if $\Lambda$ 
is also a Frobenius $\k$-algebra then $R(\A,V)$ is stable under syzygies.  We use these facts to determine the universal deformation rings of string $\Ar$-modules whose 
stable endomorphism ring isomorphic to $\k$, where $\Ar$ is a symmetric special biserial $\k$-algebra that has quiver with relations depending on the four parameters $
\bar{r}=(r_0,r_1,r_2,k)$ with $r_0,r_1,r_2\geq 2$ and $k\geq 1$.  
\end{abstract}
\section{Introduction}
Let $\k$ be a field of arbitrary characteristic, and let denote by $\hat{\Ca}$ the category of all complete local commutative Noetherian $\k$-algebras with residue field $\k$. Suppose 
that $\A$ is a fixed finite dimensional $\k$-algebra and let $V$ be a finitely generated $\A$-module.  Let $R$ be 
an arbitrary object in $\hat{\Ca}$. A {\it lift} $(M,\phi)$ of $V$ over $R$ is a finitely generated $R\otimes_\k\A$-module $M$ 
that is free over $R$ together with an isomorphism of $\A$-modules $\phi:\k\otimes_RM\to V$. If $\A$ is self-injective and the stable endomorphism ring of $V$ is isomorphic to $\k$, 
then there exists a particular object  $R(\A,V)$ in $\hat{\Ca}$ and a lift $(U(\A,V),\phi_{U(\A,V)})$ of $V$ over $R(\A,V)$, which is universal with respect to all isomorphism classes of 
lifts of $V$ over such $\k$-algebras $R$ (see \cite{blehervelez} and \S \ref{section2}). The ring $R(\A,V)$ and the isomorphism class of the lift $(U(\A,V),\phi_{U(\A,V)})$ are 
respectively called the {\it universal deformation ring} and the {\it universal deformation} of $V$. 
Traditionally, universal deformations rings are studied when $\A$ is equal to a group algebra $\k G$, where $G$ is a finite group and $\k$ has positive characteristic $p$ (see e.g., 
\cite{bleher1,bleher2,bleher3,bleher4,bleher5,bleher6,bleher7}).  In particular, it was proved by {\sc F. M. Bleher} and {\sc T. Chinburg} in \cite{bleher2} that if $V$ is a finitely generated $\k G$-module whose stable endomorphism ring is isomorphic to $\k$, then $V$ has a universal deformation ring $R(G,V)$. Observe that $\k G$ is an example of a self-injective $\k$-algebra (see e.g., \cite[Prop. 3.1.2]{benson}).  This approach has recently led to the solution of various open problems, e.g., the construction of representations 
whose universal deformation rings are not local complete intersections (see \cite{bleher1,bleher3,bleher4}). Universal deformation rings of modules over more general finite 
dimensional algebras have been studied by many authors in different contexts (see e.g., \cite{ile,yau} and their references). The main motivation of this article is that sophisticated 
results from representation theory of finite dimensional algebras,  such as Auslander-Reiten quivers, stable equivalences, and combinatorial description of modules can be used to 
arrive at a deeper understanding of universal deformation rings.

In this article,  we assume that $\k$ is algebraically closed and consider the basic $\k$-algebra 
\begin{equation}\label{algebra}
\Ar=\k Q/I_{\bar{r}}
\end{equation}
where $\bar{r}=(r_0,r_1,r_2,k)$ with $r_0,r_1,r_2\geq 2$, $k\geq 1$, $Q$ is the quiver
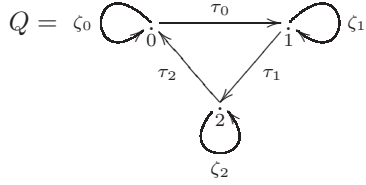
\begin{equation*}\label{quiver1}
Q= \xymatrix@1@=20pt{
	\ar@(ul,dl)_{\zeta_0} \underset{0}{.}\ar[rr]^{\tau_0}& & \underset{1}{.}\ar[dl]^{\tau_1}\ar@(ur,dr)^{\zeta_1}\\
	&\underset{2}{.}\ar[ul]^{\tau_2}\ar@(dl,dr)_{\zeta_2}&\\
	}
\end{equation*}
and $I_{\bar{r}}$ is the ideal of the path algebra $\k Q$ generated by the relations 
\begin{equation*}\label{ideal1}
 \{\tau_0\zeta_0, \zeta_1\tau_0,\tau_1\zeta_1,\zeta_2\tau_1,\tau_2\zeta_2,\zeta_0\tau_2,\zeta_0^{r_0}-(\tau_2\tau_1\tau_0)^k,\zeta_1^{r_1}-(\tau_0\tau_2\tau_1)^k, \zeta_2^{r_2}-
(\tau_1\tau_0\tau_2)^k\}.
\end{equation*}

The algebra $\Ar$ is among the class of algebras of dihedral type, which were introduced by {\sc K. Erdmann} in \cite{erdmann} to classify all tame blocks of group algebras of finite groups with dihedral 
defect groups up to Morita equivalence. However, $\Ar$ is not Morita equivalent to a block of a group algebra (see \cite[Lemma IX.5.4]{erdmann}). Since $\Ar$ is a special biserial 
algebra, all the 
non-projective indecomposable $\Ar$-modules can be described combinatorially as so-called strings and bands modules as introduced in \cite{buri} (see also \S \ref{ape1}). 
We denote by $\Gamma_s(\Ar)$ the stable Auslander-Reiten quiver of $\Ar$. The components of $\Gamma_s(\Ar)$ consisting in 
string modules are two $3$-tubes and infinitely many components of type $\mathbb{ZA}_\infty^\infty$. The components consisting of band modules are infinitely many $1$-tubes. 

In \cite{blehervelez}, the particular case $\bar{r}=(2,2,2,1)$ has being considered. In particular, there are exactly three components $\mathfrak{C}$ of $\Gamma_s(\Ar)$ of type $\mathbb{ZA}_\infty^\infty$, which each contain a simple $\Ar$-module. If $\mathfrak{C}$ is such a component  then 
$\Omega(\mathfrak{C})=\mathfrak{C}$ and there are exactly three $\Omega$-orbits of $\A_{(2,2,2,1)}$-modules in $\mathfrak{C}$ whose stable endomorphism ring is isomorphic to $
\k$; the universal deformation rings are either isomorphic to $\k$, or to $\k[[t]]/(t^2)$, or to $\k[[t]]$ (see \cite[Prop. 3.9]{blehervelez}). Moreover, if $\mathfrak{T}$ is one $3$-tubes of $
\Gamma_s(\A_{(2,2,2,1)})$ then $\Omega(\mathfrak{T})$ is the other $3$-tube and there are exactly three $\Omega$-orbits of $\A_{(2,2,2,1)}$-modules in $\mathfrak{T}$ whose stable 
endomorphism ring is isomorphic to $\k$; the universal deformation rings are either isomorphic to $\k$ or to $\k[[t]]$ (see \cite[Prop. 3.11]{blehervelez}).

In this article, we let $\bar{r}=(r_0,r_1,r_2,k)$ with $r_0,r_1,r_2\geq 2$ and $k\geq 1$ be arbitrary. We study the two $3$-tubes and the components $\mathfrak{C}$ of $\Gamma_s(\Ar)
$ of type $\mathbb{ZA}_\infty^\infty$ containing a module whose endomorphism ring is isomorphic to $\k$. Our goal is to investigate how universal deformation rings change when 
inflating modules from $\A_{(r_0,r_1,r_2,k)}$ to $\A_{(r_0',r_1',r_2',k')}$, where $\A_{(r_0',r_1',r_2',k')}$ 
surjects onto $\A_{(r_0,r_1,r_2,k)}$ when $r_0'\ge r_0$, $r_1'\ge r_1$, $r_2'\ge r_2$, $k'\ge k$. 

If $M$ and $N$ are two indecomposable $\Ar$-modules belonging to the same component of $\Gamma_s(\Ar)$, we say that $N$ is a {\it successor} of $M$ provided that there exists an irreducible homomorphism $M\to N$. Throughout this article, we identify the vertices of the quiver $Q$ with elements of the cyclic group with three elements $\mathbb{Z}/3$.

A summary of the main results concerning $\Ar=\k Q/I_{\bar{r}}$ is as follows (cf. \cite[Prop. 3.9, Prop. 3.11]{blehervelez}); for more precise statements, see Propositions \ref{prop4}, 
\ref{prop5}, \ref{prop6} and \ref{prop7}.

\begin{theorem}\label{mainthm}
Let $\Ar$ be as in (\ref{algebra}) where $\bar{r}=(r_0,r_1,r_2,k)$ with $r_0,r_1,r_2\geq 2$ and $k\geq 1$, and let $\Gamma_s(\Ar)$ denote the stable Auslander-
Reiten quiver of $\Ar$.
\begin{enumerate}
\item If $\mathfrak{T}$ is one of two the $3$-tubes then $\Omega(\mathfrak{T})$ is the other $3$-tube. There are exactly three $\Omega$-orbits of 
modules in $\mathfrak{T}\cup\Omega(\mathfrak{T})$ whose stable endomorphism ring is isomorphic to $\k$. If $X_0$ is a module that belongs to the boundary of $\mathfrak{T}$, 
then these three $\Omega$-orbits are represented by $X_0$, by a successor $X_1$ of $X_0$, and by a successor $X_2$ of $X_1$ that does not lie in the $\Omega$-orbit of 
$X_0$. The universal deformation rings are 
\begin{align*}
R(\Ar,X_0)\cong \k,&& R(\Ar,X_1)\cong \k, && R(\Ar,X_2)\cong \k[[t]]. 
\end{align*}
\item There are exactly three distinct components $\mathfrak{A}_0$, $\mathfrak{A}_1$ and  $\mathfrak{A}_2$ of $\Gamma_s(\Ar)$ of type $\mathbb{ZA}_\infty^\infty$, which each 
contain exactly one simple $\Ar$-module. For all  $i\in \{0,1,2\}\mod 3$, the component $\mathfrak{A}_i$ is $\Omega$-stable if and only if $r_i=2$ and there are exactly three $\Omega
$-orbits of modules in $\mathfrak{A}_i\cup\Omega(\mathfrak{A}_i)$ whose stable endomorphism ring is isomorphic to $\k$. If for all $i\in \{0,1,2\}\mod 3$, $U_{i,0}$ denotes the unique 
simple module lying in $\mathfrak{A}_i$, then these three $\Omega$-orbits are represented by $U_{i,0}$, by a successor $U_{i,1}$ of $U_{i,0}$, and by a successor $U_{i,2}$ of 
$U_{i,1}$ that does not lie in the $\Omega$-orbit of $U_{i,0}$. The universal deformation rings are
\begin{align*}
R(\Ar,U_{i,0})\cong \k[[t]]/(t^{r_i}),&& R(\Ar,U_{i,1})\cong \k, && R(\Ar,U_{i,2})\cong \k[[t]]. 
\end{align*}
\item There are three distinct components $\mathfrak{B}_0$, $\mathfrak{B}_1$ and $\mathfrak{B}_2$  of $\Gamma_s(\Ar)$ of type $\mathbb{ZA}_\infty^\infty$ that contain exactly a 
module of length $1$ whose endomorphism ring is isomorphic to $\k$. Let $i\in \{0,1,2\}\mod 3$ and let $V_{i,0}$ be a module of minimal length in $\mathfrak{B}_i$. If $k=1$ then $
\mathfrak{B}_i=\Omega(\mathfrak{A}_{i+2})$,  where $\mathfrak{A}_{i+2}$ is as in $\mathrm{(ii)}$. In particular, $\mathfrak{B}_i$ is $\Omega$-stable if and only if $k=1$ and $r_{i+2}
=2$. There are exactly three $\Omega$-orbits of modules in $\mathfrak{B}_i\cup\Omega(\mathfrak{B}_i)$ whose stable endomorphism ring is isomorphic to $\k$. These three $
\Omega$-orbits are represented by $V_{i,0}$, by a successor $V_{i,1}$ of $V_{i,0}$, and by a successor $V_{i,-1}$ of $V_{i,0}$ that does not lie in the $\Omega$-orbit of $V_{i,1}$. If 
$k=1$ then the universal deformation rings are
\begin{align*}
R(\Ar,V_{i,0})\cong \k,&& R(\Ar,V_{i,1})\cong \k[[t]]/(t^{r_{i+2}}), && R(\Ar,V_{i,-1})\cong \k[[t]]. 
\end{align*}
If $k\geq 2$ then the universal deformation rings are
\begin{align*}
R(\Ar,V_{i,0})\cong \k,&& R(\Ar,V_{i,1})\cong \k[[t]], && R(\Ar,V_{i,-1})\cong \k[[t]]. 
\end{align*}

\item If $k\geq2$ then there are three distinct components $\mathfrak{C}_0$, $\mathfrak{C}_1$ and $\mathfrak{C}_2$ of $\Gamma_s(\Ar)$ of type $\mathbb{ZA}_\infty^\infty$, which each contain a module of length $2$ whose endomorphism rings is isomorphic to $\k$. Let $i\in\{0,1,2\}\mod 3$ and let $W_{i,0}$ be a module of minimal length in $\mathfrak{B}_i$. 
For all $i\in\{0,1,2\}\mod 3$, the component $\mathfrak{C}_i$ is $\Omega$-stable if and only if $k=2$, and there are exactly three $\Omega$-orbits of modules in $\mathfrak{C}_i$ 
whose stable endomorphism ring is isomorphic to $\k$. These three $\Omega$-orbits are represented by $W_{i,0}$, by a 
successor $W_{i,-1}$ of $W_{i,0}$, and by a successor $W_{i,-2}$ of $W_{i,-1}$ that does not lie in the $\Omega$-orbit of $W_{i,0}$. The universal deformation rings are
\begin{align*}
R(\Ar,W_{i,0})\cong \k[[t]]/(t^k),&& R(\Ar,W_{i,-1})\cong \k, && R(\Ar,W_{i,-2})\cong \k[[t]]. 
\end{align*}
\end{enumerate}
\end{theorem} 

This article is organized as follows. In \S \ref{section2}, we recall the definitions of deformations and universal deformation rings and summarize some of their properties. In \S 
\ref{section3}, we give a precise description of string modules for $\Ar$, describe the components of $\Gamma_s(\Ar)$ that contain string modules using hooks and co-hooks (see \cite{buri}), and give a description of the homomorphisms between strings modules as determined in \cite{krause}. Moreover, we describe the indecomposable projective modules for $\Ar$ and classify all $\Ar$-modules with endomorphism ring isomorphic to $\k$ (see Proposition \ref{prop2}). In \S \ref{section4}, we prove Theorem \ref{mainthm}.

See e.g., \cite{auslander,benson,erdmann} for further information about basic concepts from representation theory of finite dimensional algebras, such as the definition and properties of the syzygy functor $\Omega$ and the definition of the Auslander-Reiten quiver of an arbitrary Artinian algebra $\A$. 

\section{Universal Deformation Rings}
\label{section2}
Let $\k$ be a a field of arbitrary characteristic and denote by $\hat{\Ca}$ the category of all complete local commutative Noetherian $\k$-algebras with residue field $\k$. Note that 
the morphisms in $\hat{\Ca}$ are continuous $\k$-algebra homomorphisms that induce the identity map on $\k$.  Suppose that $\A$ is a finite dimensional $\k$-algebra and 
$V$ is a fixed finitely generated $\A$-module. We denote by $\End_\A(V)$ (respectively, by $\SEnd_\A(V)$) the endomorphism ring (respectively, the stable endomorphism ring) of $V
$. Let $R$ be an arbitrary object in $\hat{\Ca}$. A {\it lift} $(M,\phi)$ of $V$ over $R$ is a finitely generated $R\otimes_\k\A$-module $M$ 
that is free over $R$ 
together with an isomorphism of $\A$-modules $\phi:\k\otimes_RM\to V$. Two lifts $(M,\phi)$ and $(M',\phi')$ over $R$ are {\it isomorphic} if there exists an $R\otimes_\k\A$-module 
isomorphism $f:M\to M'$ such that $\phi'\circ (\text{id}_\k\otimes f)=\phi$, where $\text{id}_\k$ denotes the identity map on $\k$.
If $(M,\phi)$ is a lift of $V$ over $R$ we  denote by $[M,\phi]$ its isomorphism class and say that $[M,\phi]$ is a {\it deformation} of $V$ over $R$. We denote by $\Def_\A(V,R)$ the 
set of all deformations of $V$ over $R$. The {\it deformation functor } over $V$ is the 
covariant functor $\hat{\Fun}_V:\hat{\Ca}\to \Sets$ defined as follows: for all objects $R$ in $\hat{\Ca}$ define $\hat{\Fun}_V(R)=\Def_\A(V,R)$ and for all morphisms $\alpha:R\to 
R'$ in 
in $\hat{\Ca}$ 
let $\hat{\Fun}_V(\alpha):\Def_\A(V,R)\to \Def_\A(V,R')$ be defined as $\hat{\Fun}_V(\alpha)([M,\phi])=[R'\otimes_{R,\alpha}M,\phi_\alpha]$, where $\phi_\alpha: \k\otimes_{R'}
(R'\otimes_{R,\alpha}M)\to V$ is the composition of $\A$-module isomorphisms 
\[\k\otimes_{R'}(R'\otimes_{R,\alpha}M)\cong \k\otimes_RM\xrightarrow{\phi} V.\]  

Suppose there exists an object $R(\A,V)$ in $\hat{\Ca}$  and a deformation $[U(\A,V), \phi_{U(\A,V)}]$ of $V$ over $R(\A,V)$ with the following property. For each $R$ in $\hat{\Ca}$ 
and for all lifts $M$ of $V$ over $R$ there exists a morphism $\upsilon:R(\A,V)\to R$ in $\hat{\Ca}$ such that 
\[\hat{\Fun}_V(\upsilon)[U(\A,V), \phi_{U(\A,V)}]=[M,\phi],\]
and moreover $\upsilon$ is unique if $R$ is the ring of dual numbers $\k[[t]]]/(t^2)$.  Then $R(\A,V)$ and $[U(\A,V),\phi_{U(\A,V)}]$ are respectively called  the {\it versal deformation 
ring} and {\it versal deformation} of $V$. If the morphism $\upsilon$ is unique for all $R$ in $\hat{\Ca}$ and lifts $(M,\phi)$ of $V$ over $R$, then $R(\A,V)$ and $[U(\A,V),
\phi_{U(\A,V)}]$ are respectively called the {\it universal deformation ring} and the {\it universal deformation} of $V$.  In other words, the universal deformation ring $R(\A,V)$ 
represents the deformation functor $\hat{\Fun}_V$ in the sense that $\hat{\Fun}_V$ is naturally isomorphic to the $\Hom$ functor $\Hom_{\hat{\Ca}}(R(\A,V),-)$. Using Schlessinger's 
criteria \cite[Thm. 2.11]{sch} and using methods similar to those in \cite{mazur}, it is straightforward to prove that the deformation functor $\hat{\Fun}_V$ is continuous, that every finitely 
generated $\A$-module $V$ has a versal deformation ring and that this versal deformation is universal provided that the endomorphism ring $\End_\A(V)$ is isomorphic to $\k$ (see 
\cite[Prop. 2.1]{blehervelez}). 

Recall that the $\k$-algebra $\A$ is said to be self-injective if the regular left $\A$-module $_\A\A$ is injective and that $\A$ is called a Frobenius algebra provided that the right $\A$-
modules $\A_\A$ and $({_\A}\A)^\ast=\Hom_\k({_\A}\A,\k)$ are isomorphic. Recall also that $\A$ is said to be a symmetric algebra provided that $\A$ is a Frobenius algebra and 
there exists a non-degenerate associative bilinear form $\theta: \A\times \A\to \k$ with $\theta(a,b)=\theta(b,a)$ for all $a,b\in \A$. By \cite
[Prop. 9.9]{curtis}, every Frobenius algebra is self-injective. 

\begin{remark}\label{rem1}
If $\A$ is self-injective and $(M,\phi)$ is a lift of $V$ 
over an object $R$ in $\hat{\Ca}$ with $\SEnd_\A(V)\cong \k$, then the deformation $[M,\phi]$ does not depend on the particular choice of the $\A$-module isomorphism. More 
precisely, if $f:M\to M'$ is an $R\otimes_\k \A$-module isomorphism with $(M',\phi')$ a lift of $V$ over $R$, then there exists an $R\otimes_\k\A$-module isomorphism $\bar{f}:M\to M'$ 
such that $\phi'\circ (\mathrm{id}_\k\otimes _R\bar{f})=\phi$. In other words, $[M,\phi]=[M',\phi']$ in $\hat{\Fun}_V(R)=\Def_\A(V,R)$ 
(see \cite[Thm. 2.6] {blehervelez}). 
\end{remark}

We denote the first syzygy of  $V$ by $\Omega V$, i.e., $\Omega V$ is the kernel of a projective cover $P_V\to V$, (see e.g., \cite[pp. 124-126]{auslander}).  
\begin{example}\label{exam1}
Let $G$ be a finite group and consider the group algebra $\k G$, which is a self-injective $\k$-algebra (see e.g., \cite[Prop. 3.1.2]{benson} and \cite[Prop. 9.6]{curtis}).  It was proved in 
\cite{bleher2} that if $V$ is a finitely generated $\k G$-module whose stable endomorphism ring is isomorphic to $\k$ then $V$ has a universal deformation ring $R(\k G, V)$. Moreover, 
the stable 
endomorphism ring of $\Omega V$ is also isomorphic to $\k$ and  the universal deformation rings $R(\k G,V)$ and  
$R(\k G,\Omega V)$ of $V$ and $\Omega V$, respectively, are isomorphic.
\end{example}

The following result generalizes the properties of universal deformation rings mentioned in Example \ref{exam1} to arbitrary Frobenius $\k$-algebras (see \cite[Thm. 2.6]
{blehervelez}).

\begin{theorem}\label{thm3}
Let $\A$ be a finite dimensional self-injective $\k$-algebra, and suppose that $V$ is a finitely generated $\A$-module whose stable endomorphism ring $\SEnd_\A(V)$ is isomorphic 
to $\k$.
\begin{itemize}
\item[\textup{(i)}] The module $V$ has a universal deformation ring $R(\A,V)$.
\item[\textup{(ii)}] If $P$ is a finitely generated projective $\A$-module, then $\SEnd_\A(V\oplus P)\cong \k$ and $R(\A,V)\cong R(\A,V\oplus P)$.
\item[\textup{(iii)}] If $\A$ is also a Frobenius algebra, then $\SEnd_\A(\Omega V)\cong \k$ and $R(\A,V)\cong R(\A,\Omega V)$. 
\end{itemize}
\end{theorem}

\section{Some Remarks about the Representation Theory of $\Ar$ and classification of $\Ar$-modules whose endomorphism ring is isomorphic to $\k$}\label{section3}
For the remainder of this article, let $\k$ be an algebraically closed field of arbitrary characteristic and let $\Ar=\k Q/I_{\bar{r}}$ as in (\ref{algebra}). We identify the vertices of $Q$ with 
elements of  $\mathbb{Z}/3$ (the cyclic group of three elements).

The algebra $\Ar$ is one of the algebras of dihedral type studied by {\sc K. Erdmann} in \cite{erdmann}. In particular, $\Ar$ is a symmetric $\k$-
algebra. However, by \cite[Lemma IX.5.4]{erdmann}, $\Ar$ is not Morita equivalent to a block of a group algebra. Since $\Ar$ is a special biserial algebra, all the 
non-projective indecomposable $\Ar$-modules can be described combinatorially as so-called strings and bands modules (see \cite{buri}). In this article, we are only concerned about 
these string modules, which are described as follows.

\subsection{String modules for $\Ar$}
\label{ape1}
Given each arrow $\zeta_0, \tau_0,\zeta_1,\tau_1,\zeta_2,\tau_2$ of $Q$, we define a formal inverse by $\zeta_0^{-1}$, $\tau_0^{-1}$, $\zeta_1^{-1}$, $\tau_1^{-1}$, $\zeta_2^
{-1}$, $\tau_2^{-1}$, 
respectively. Let $\text{\bf s}(\zeta_0)=0=\text{\bf s}(\zeta_0^{-1})$, $\text{\bf s}(\tau_0)=0=\text{\bf s}(\tau_2^{-1})$, $\text{\bf s}(\zeta_1)=1=\text{\bf s}(\zeta_1^{-1})$, $\text{\bf s}
(\tau_1)=1=
\text{\bf s}(\tau_0^{-1})$, $\text{\bf s}(\zeta_2)=2=\text{\bf s}(\zeta_2^{-1})$ and $\text{\bf s}(\tau_2)=2=\text{\bf s}(\tau_1^{-1})$. Let $\text{\bf e}(\zeta_0)=0=\text{\bf e}(\zeta_0^{-1})$, $
\text{\bf e}
(\tau_0)=1=\text{\bf e}(\tau_1^{-1})$, $\text{\bf e}(\zeta_1)=1=\text{\bf e}(\zeta_1^{-1})$, $\text{\bf e}(\tau_1)=2=\text{\bf s}(\tau_2^{-1})$, $\text{\bf e}(\zeta_2)=2=\text{\bf e}(\zeta_2^
{-1})$ and $\text{\bf e}
(\tau_2)=0=\text{\bf e}(\tau_0^{-1})$. By a {\it word} of length $n\geq 1$ we mean a sequence $w_n\cdots w_1$, where the $w_j$ is either an arrow or a formal inverse of an arrow 
and where $\text{\bf s}(w_{j+1})=\text{\bf e}(w_j)$ for  $1\leq j \leq n-1$. We define $(w_n\cdots w_1)^{-1}={w_1}^{-1}\cdots {w_n}^{-1}$, $\text{\bf s}(w_n\cdots w_1)=\text{\bf s}
(w_1)$ and $\text{\bf e}(w_n\cdots w_1)=\text{\bf e}(w_n)$.  If $i\in \{0,1,2\}\mod 3$ is a vertex of $Q$, we define an empty word $\1_i$ of length zero with $\text{\bf e}(\1_i)=i=\text{\bf s}
(\1_i)$ and $
(\1_i)^{-1}=\1_i$. 
For all $i\in \{0,1,2\}\mod 3$ there exists a path in $Q$ of length $3$ starting and ending at $i$, namely   
\begin{equation}\label{projec}
\underline{c}_i=\tau_{i+2}\tau_{i+1}\tau_i.
\end{equation}
Denote by $\mathcal{W}$ the set of all words and let 
\[J=\{\zeta_0^{r_0},\zeta_1^{r_2},\zeta_2^{r_3}, \tau_0\zeta_0,\zeta_1\tau_0,\tau_1\zeta_1, 
\zeta_2\tau_1,\tau_2\zeta_2,\zeta_0\tau_2, \underline{c}_0^k, \underline{c}_1^k, \underline{c}_2^k\}.\] 
Let $\sim$ be the equivalence relation on $\mathcal{W}$ defined by  $w\sim w'$ if and only if $w=w'$ or $w^{-1}
=w'$. A \textit{string} is a representative $C$ of an equivalence class under the relation $\sim$  where either  $C=\1_i$ for some vertex $i$ of $Q$, or $C=w_n\cdots w_1$ with $n
\geq 1$ and $w_j\not=w_{j+1}^{-1}$ for $1\leq j\leq  n-1$ and no sub-word of $C$ or its formal inverse belong to $J$. If $C$ is a string such that ${\bf s}(C)={\bf e}(C)$, then we let 
$C^0=\1_{{\bf e}(C)}$. If $C=w_n\cdots w_1$ and $D=v_m\cdots v_1$ are strings of 
length $n,m\geq 1$, respectively, we say that the composition of $C$ and $D$ is defined provided that $w_n\cdots w_1v_m\cdots v_1$ is a string, and write $CD=w_n\cdots 
w_1v_m\cdots v_1$; we say that the composition of $C$ with $\1_i$ is defined provided $\text{\bf s}(C)=i$ (respectively, $\text{\bf e}(C)=i$), and in this case we have $C
\1_i\sim C$ (respectively, $\1_iC \sim C$). In particular, if $C=w_n\cdots w_1$ is a string of length $n\geq 1$ then $C\sim w_n\cdots w_{j+1}\1_{\text{\bf e}(w_j)}w_j\cdots w_1$ for 
all $1\leq j
\leq n-1$. 
If $C=w_n\cdots w_1$ is a string of length $n\geq 1$ then there exists and indecomposable $\Ar$-module $M[C]$, called the {\it string 
module} corresponding to the string representative $C$, which can be  described as follows. There is an ordered $\k$-basis $\{z_0,z_1,\ldots,z_n\}$ of $M[C]$ such that the action of $
\Ar$ 
on $M[C]$ is given by the following representation $\varphi_C:\Ar\to \Mat(n+1,\k)$. Let $\text{\bf v}(j)=\text{\bf e}(w_j)$ for $0\leq j\leq n-1$ and $\text{\bf v}(n)=\text{\bf s}
(w_n)
$. Then for each vertex $i\in \{0,1,2\}\mod3$ and for each arrow $\zeta\in \{\zeta_0, \tau_0,\zeta_1, \tau_1,\zeta_2,\tau_2\}$ in $Q$ and for all $0\leq j\leq n$ define
\begin{align}
\varphi_C(i)(z_j)=\begin{cases}
z_j, &\text{ if $\text{\bf v}(j)=i$}\\
0, &\text{ otherwise}
\end{cases}
&&
\text{ and }
&&
\varphi_C(\zeta)(z_j)=\begin{cases}
z_{j-1}, &\text{ if $w_j=\zeta$}\\
z_{j+1}, &\text{ if $w_{j+1}=\zeta^{-1}$}\\
0, &\text{ otherwise}
\end{cases}
\end{align}

We call $\varphi_C$ the {\it canonical representation} and $\{z_0,z_1,\ldots,z_n\}$ a {\it canonical $\k$-basis} for $M[C]$ relative to the string representative $C$. Note that $M[C]
\cong M
[C^{-1}]$.  If $C=\1_i$ with $i\in \{0,1,2\}\mod 3$ 
then $M[C]$ is the simple $\Ar$-module corresponding to vertex $i$. We denote the simple $\Ar$-modules corresponding to the vertices $0$, $1$ and $2$ of $Q$ 
by $M[\1_0]$, $M[\1_1]$ and $M[\1_2]$, 
respectively.

\subsection{The stable Auslander-Reiten quiver of $\Ar$}
\label{ape2}

We denote by $\Gamma_s(\Ar)$ the stable Auslander-Reiten quiver of $\Ar$ (see \cite[VII]{auslander}).
For all $i\in \{0,1,2\}\mod 3$ there exists a path in $Q$ of length $3$ starting and ending at $i$, namely   
\begin{equation}\label{projec}
\underline{c}_i=\tau_{i+2}\tau_{i+1}\tau_i.
\end{equation}

The components $\mathfrak{C}$ of  $\Gamma_s(\Ar)$ consisting of string $\Ar$-modules are two $3$-tubes and infinitely many non-periodic components of type $\mathbb{ZA}_
\infty^\infty$.  In the following, we describe the irreducible morphism between string $\Ar$-modules. 

Assume $C=w_nw_{n-1}\cdots w_1$ with $n\geq 1$ is a string. We say that $C$ 
is \textit{directed} if all $w_j$ are arrows 
and  we say that $C$ is a {\it maximal directed string} if $C$ is directed and if for any arrow $\zeta$ in $Q$, $\zeta C\in J$.
Let $\mathcal{M}$ be the set of all maximal directed strings, i.e.,
\[\mathcal{M}=\{\zeta_0^{r_0-1},\zeta_1^{r_1-1},\zeta_2^{r_2-1},\underline{c}_0^{k-1}\tau_2\tau_1,
\underline{c}_1^{k-1}\tau_0\tau_2,\underline{c}_2^
{k-1}\tau_1\tau_0\}.\]

Let $C$ be a string. We say that $C$ {\it starts on a peak} (respectively, {\it starts in a deep}) provided that there is no arrow $\zeta$ in $Q$ such that $C\zeta$ (respectively,  $C
\zeta^{-1}$) is a string;  we also say that $C$ {\it ends 
in a peak} (respectively, {\it ends in a deep}) provided that there is no arrow $\gamma$ in $Q$ such that  $\gamma^{-1}C$ (respectively, $\gamma C$) is a string. If $C$ is a string 
not starting on a peak (respectively, not starting in a deep) , say $C\zeta$ (respectively, $C\zeta^{-1}$) is a string for some arrow $\zeta$ then there is a unique directed string $D\in 
\mathcal{M}$ such that $C_h=C\zeta D^{-1}$  (respectively, $C_c=C\zeta^{-1}D$) is a string. We say $C_h$ (respectively, $C_c$) is obtained from $C$ by adding a {\it hook} 
(respectively, a {\it co-hook}) on the right side.
Dually,  if $C$ is a string not ending on a peak (respectively, not ending in a deep), say $\gamma^{-1} C$ (respectively, $\gamma C$) is a string for some arrow $\gamma$ in $Q$ 
then there is a unique directed string $E\in \mathcal{M}$ such that $_hC=E\gamma^{-1} C$  (respectively, $_cC=E^{-1}\gamma D$) is a string. We say $_hC$ (respectively, $_cC$) 
is obtained from $C$ by adding a {\it hook} (respectively, a {\it co-hook}) on the left side. By \cite{buri}, all irreducible morphisms between string modules are either canonical 
injections $M[C]\to M[C_h]$, $M[C]\to M[{_hC}]$, or canonical projections $M[C_c]\to M[C]$, $M[{_cC}]\to M[C]$. 
Suppose $M[C]$ is a string module of minimal length such that $M[C]$ belongs to a component $\mathfrak{C}$ of $\Gamma_s(\Ar)$ of type $\mathbb{ZA}_\infty^
\infty$. 
\begin{figure}[ht]
$
\begindc{\commdiag}[10]
\obj(-12,10)[a]{$\vdots$}
\obj(-4,10)[1]{$\vdots$}
\obj(4,10)[2]{$\vdots$}
\obj(12,10)[b]{$\vdots$}
\obj(-12,7)[3]{$\cdots$}
\obj(-8,7)[4]{$M[C_{cc}]$}
\obj(0,7)[5]{$M[_hC_{c}]$}
\obj(8,7)[6]{$M[_{hh}C]$}
\obj(12,7)[7]{$\cdots$}

\obj(-12,4)[8]{$\cdots$}
\obj(-4,4)[9]{$M[C_{c}]$}
\obj(4,4)[10]{$M[_hC]$}
\obj(12,4)[11]{$\cdots$}

\obj(-12,1)[12]{$\cdots$}
\obj(-8,1)[13]{$M[_cC_{c}]$}

\obj(0,1)[14]{$M[C]$}
\obj(8,1)[15]{$M[_hC_h]$}
\obj(12,1)[16]{$\cdots$}

\obj(-12,-2)[17]{$\cdots$}
\obj(-4,-2)[18]{$M[_cC]$}
\obj(4,-2)[19]{$M[C_h]$}
\obj(12,-2)[20]{$\cdots$}

\obj(-12,-5)[21]{$\cdots$}
\obj(-8,-5)[22]{$M[_{cc}C]$}
\obj(0,-5)[23]{$M[_cC_h]$}
\obj(8,-5)[24]{$M[C_{hh}]$}
\obj(12,-5)[25]{$\cdots$}

\obj(-12,-8)[26]{$\vdots$}
\obj(-4,-8)[27]{$\vdots$}
\obj(4,-8)[28]{$\vdots$}
\obj(12,-8)[29]{$\vdots$}

\mor{a}{4}{}
\mor{4}{1}{}
\mor{1}{5}{}
\mor{5}{2}{}
\mor{2}{6}{}
\mor{6}{b}{}
\mor{8}{4}{}
\mor{4}{9}{}
\mor{9}{5}{}
\mor{5}{10}{}
\mor{10}{6}{}
\mor{6}{11}{}
\mor{8}{13}{}
\mor{13}{9}{}
\mor{9}{14}{}
\mor{14}{10}{}
\mor{10}{15}{}
\mor{15}{11}{}
\mor{17}{13}{}
\mor{13}{18}{}
\mor{18}{14}{}
\mor{14}{19}{}
\mor{19}{15}{}
\mor{15}{20}{}
\mor{17}{22}{}
\mor{22}{18}{}
\mor{18}{23}{}
\mor{23}{19}{}
\mor{19}{24}{}
\mor{24}{20}{}
\mor{26}{22}{}
\mor{22}{27}{}
\mor{27}{23}{}
\mor{23}{28}{}
\mor{28}{24}{}
\mor{24}{29}{}
\enddc
$
\caption{The stable Auslander-Reiten component near $M[C]$.}\label{component}
\end{figure}
Since none of the projective $\Ar$-modules is uniserial then near $M[C]$ the component $\mathfrak{C}$ looks as in Figure \ref{component}.

\subsection{Homomorphisms between string modules for $\Ar$}
\label{ape3}
Let $S$ and $T$ be strings for $\Ar$. Suppose $C$ is a substring of both $S$ and $T$ such that the following conditions (i) and (ii) are satisfied.
\begin{itemize}
\item[(i)] $S\sim BCD$, where $B$ is a substring which is either of length zero or $B=B'\zeta$ for an arrow $\zeta$, and $D$ is a substring which is either of length zero or $D=
\gamma^{-1}D'$ for an arrow $\gamma$, i.e., $S\sim B'\xleftarrow{\tau}C\xrightarrow{\varphi} D'$.
\item[(ii)] $T\sim ECF$, where $E$ is a substring which is either of length zero or $E=E'\epsilon^{-1}$ for an arrow $\epsilon$, and $F$ is a substring which is either of length zero or 
$F=\mu F'$ for an arrow $\mu$, i.e., $T\sim E'\xrightarrow{\epsilon} C\xleftarrow{\mu} F'$.
\end{itemize}
Then by \cite{krause} there exists a composition of  $\Ar$-module homomorphisms
\begin{equation}\label{canhom}
\sigma_C:M[S]\surjection M[C]\injection M[T].
\end{equation}
We call $\sigma_C$ a {\it canonical homomorphism} from $M[S]$ to $M[T]$ that factors through $M[C]$. It follows from \cite{krause} that each $\Ar$-module homomorphism 
from $M[S]$ to $M[T]$ can be written uniquely as a $\k$-linear combination of canonical $\Ar$-module homomorphisms as in (\ref{canhom}). In particular, if $M[S]=M[T]$ 
then the canonical endomorphisms generate $\End_{\Ar}(M[S])$.

\subsection{Projective Indecomposable $\Ar$-modules and modules whose endomorphism ring is isomorphic to $\k$}
\label{3}
For all $i\in \{0,1,2\}\mod 3$ vertex of $Q$, the radical series of the projective indecomposable $\Ar$-module $P_i$ can be described as in the following figure.
\begin{align}\label{projec}
P_i=\begindc{\commdiag}[15]
\obj(-3,5)[v1]{$M[\1_i]$}
\obj(-6,3)[v2]{$M[\1_{i+1}]$}
\obj(-6,0)[v3]{$M[\1_{i+2}]$}
\obj(-6,-3)[v4]{$M[\1_i]$}
\obj(-3,-5)[v5]{$M[\1_i]$}
\obj(0,3)[v6]{$M[\1_i]$}
\obj(0,-3)[v7]{$M[\1_i]$}
\mor{v1}{v2}{$\tau_i$}[-1,0]
\mor{v2}{v3}{$\tau_{i+1}$}[-1,0]
\mor{v3}{v4}{$\tau_{i+2}$}[-1,0]
\mor{v4}{v5}{$\underline{c}_i^{k-1}$}[-1,1]
\mor{v1}{v6}{$\zeta_i$}[1,0]
\mor{v6}{v7}{$\zeta_i^{r_i-2}$}[1,1]
\mor{v7}{v5}{$\zeta_i$}[1,0]
\enddc
\end{align}

The following result provides a classification of all $\Ar$-modules whose endomorphism ring is isomorphic to $\k$. 

\begin{proposition}\label{prop2}
Let $M[S]$ be a string $\Ar$-module, where $\bar{r}=(r_0,r_1,r_2,k)$ and $r_0,r_1,r_2\geq 2$, $k\geq 1$. Then $M[S]$ has endomorphism ring isomorphic to $\k$ if and 
only if for some $i\in \{0,1,2\}
\mod 3$ the string representative $S$ is equivalent either to $\1_i$, or to $\tau_i$, or to $\tau_{i+1}\tau_i$. 
\end{proposition}
\begin{proof}
If $S$ is equivalent either to one of the strings $\1_0,\1_1$, or to $\1_2$, then it follows from Schur's Lemma that $\End_\A(M[S])\cong \k$. If $S$ is either equivalent to one of the 
strings $\tau_0$, $\tau_1$, $\tau_2$, $\tau_1\tau_0$, $\tau_2\tau_1$, or to $\tau_0\tau_2$ then the only canonical endomorphism in $\End_{\A_
{\bar{r}}}(M[S])$ is the identity homomorphism, which implies that $\End_{\Ar}(M[S])$ is one-dimensional over $\k$.
Next assume that $M[S]$ is a string $\Ar$-module with endomorphism ring isomorphic to $\k$. Let denote by $n$ the length of $S$. If $n=0$ then $S$ is equivalent 
either to  $\1_0$, or to $\1_1$, or  to $\1_2$. If $n=1$ then $S$ is equivalent to an arrow. By hypothesis, $S$ is equivalent neither to $\zeta_0$, nor to $
\zeta_1$, nor to $\zeta_2$, for otherwise $\dim_\k\End_{\Ar}(M[S])\geq 2$. This implies that $S$ is equivalent either to $\tau_0$, or to $\tau_1$,  or to $\tau_2$. For the remainder of the 
proof, assume that $n\geq 2$ and let $m$ be maximal such that  the string representative $S$ contains a substring equivalent to $\zeta_i^{-m}$ for some $i\in \{0,1,2\}\mod 3$, and put 
$m=0$ provided that $S$ does not contain as substring any of the strings $\zeta_0$, $\zeta_1$, $\zeta_2$ or any of their formal inverses.  If $m>0$ then there exist suitable strings $D
$ and $D'$ such that $S\sim D\zeta_i^{-m}D'$. It follows from the maximality of $m$ that the string $\zeta_i^{-m}$ starts in a deep and ends on a peak. Therefore, there exists a non-
trivial canonical 
endomorphism of $M[S]$ factoring through $M[\1_i]$ implying that $\dim_\k\End_{\Ar}(M[S])\geq 2$, which contradicts our hypothesis. Thus $m=0$, implying that 
$S$ does not contain as substrings the arrows $\zeta_0$, $\zeta_1$, or $\zeta_2$ or any of their formal inverses. Thus, there exist $i\in \{0,1,2\}\mod 3$ and an integer $l\in 
\{0,\ldots,k-1\}$ such that either $S\sim \underline{c}_i^l\tau_{i+2}$ or $S\sim \underline{c}_i^l\tau_{i+2}\tau_{i+1}$. 
If $l=0$ then $S$ is equivalent either to $\tau_1\tau_0$, or to $\tau_2\tau_1$, or to $\tau_0\tau_2$. Assume then that $l>0$. If $S\sim 
\underline{c}_i^l\tau_{i+2}$ (respectively, $S\sim \underline{c}_i^l\tau_{i+2}\tau_{i+1}$) then there exists a non-trivial canonical endomorphism of $M[S]$ factoring through $M[\tau_
{i+2}]$ (respectively, through $M[\tau_{i+2}\tau_{i+1}]$) implying that $\dim_\k\End_{\Ar}(M[S])\geq 2$, contradicting again our hypothesis. This finishes the proof of Proposition 
\ref{prop2}.
\end{proof}

\section{Components of $\Gamma_s(\Ar)$ of type $\mathbb{ZA}_\infty^\infty$ containing a module whose endomorphism ring is isomorphic to $\k$ and $3$-tubes}
\label{section4}
For all $i\in \{0,1,2\}\mod 3$ we define:
\begin{align}
\underline{a}_i=\tau_i\zeta_i^{-r_i+1} && \underline{b}_i=\underline{c}_{i+2}^{k-1}\tau_{i+1}\tau_i\zeta_i^{-1}
\end{align}

\subsection{Components of $\Gamma_s(\Ar)$ of type $\mathbb{ZA}_\infty^\infty$ containing a module whose endomorphism ring is isomorphic to $\k$}

\begin{proposition}\label{prop4}
For $i\in \{0,1,2\}\mod 3$, let $\mathfrak{A}_i$ be the component of the stable Auslander-Reiten quiver of $\Ar$ containing the simple $\Ar$-module $M[\1_i]$, 
where $\bar{r}=(r_0,r_1,r_2,k)$ and $r_0,r_1,r_2\geq 2$, $k\geq 1$. 
Define 
\begin{align*}
(\1_i)_h=\underline{a}_{i+2}&&\text{ and }&&(\1_i)_{hh}=\underline{a}_{i+2}\underline{a}_{i+1}.
\end{align*}
The component $\mathfrak{A}_i$ is $\Omega$-stable if and only if for $r_i=2$. If $k=1$ then the module $M[\tau_{i+1}]$ lies in $\Omega(\mathfrak{A}_i)$. The modules in $
\mathfrak{A}_i\cup \Omega(\mathfrak{A}_i)$ whose stable endomorphism rings are isomorphic to $\k$ are precisely the modules in $\Omega$-orbits of the modules $U_0=M[\1_i]$, 
$U_1=M[(\1_i)_h]$ and $U_2=M[(\1_i)_{hh}]$. Their universal deformation rings are 
\begin{align*}
R(\Ar,U_0)\cong \k[[t]]/(t^{r_i}),&& R(\Ar,U_1)\cong \k, && R(\Ar,U_2)\cong \k[[t]]. 
\end{align*}
\end{proposition}   
\begin{proof}
Let $i\in\{0,1,2\}\mod 3$ be fixed. Using hooks and co-hooks (see \S \ref{ape2}), we see that all $\Ar$-modules in $\mathfrak{A}_i\cup\Omega(\mathfrak{A}_i)$ lie in the $\Omega$-orbit 
of either
\begin{align*}
A_{q,0}&=M[(\underline{a}_{i+2}\underline{a}_{i+1}\underline{a}_i)^q], \text{ or }\\
A_{q,1}&=M[(\underline{a}_{i+2}\underline{a}_{i+1}\underline{a}_i)^q\underline{a}_{i+2}], \text{ or }\\
A_{q,2}&=M[(\underline{a}_{i+2}\underline{a}_{i+1}\underline{a}_i)^q\underline{a}_{i+2}\underline{a}_{i+1}], \text{ or }\\
B_{q,0}&=M[(\underline{b}_{i+1}\underline{b}_{i+2}\underline{b}_i)^q], \text { or }\\
B_{q,1}&=M[\underline{b}_i(\underline{b}_{i+1}\underline{b}_{i+2}\underline{b}_i)^q], \text { or }\\
B_{q,2}&=M[\underline{b}_{i+2}\underline{b}_i(\underline{b}_{i+1}\underline{b}_{i+2}\underline{b}_i)^q]
\end{align*}
for some $q\geq 0$. Note for example that $A_{0,0}=M[\1_i]=B_{0,0}$, $A_{0,1}=M[(\1_i)_h]$,  $A_{0,2}=M[(\1_i)_{hh}]$, $B_{0,1}=M[{_h(\1_i)}]$ and $B_{0,2}=M[{_{hh}(\1_i)}]$. Since 
$\Omega M[\1_i] = M[\zeta_i^{-r_i+1}\underline{c}_i^{k-1}\tau_{i+2}\tau_{i+1}]$ then $\mathfrak{A}_i=\Omega(\mathfrak{A}_i)$ if and only if $r_i=2$; and that if $k=1$ then 
$\Omega(M[\1_i])= M[{_c}(\tau_{i+1})]$. 

Using \S \ref{ape3} and the description of the projective indecomposable $\Ar$-module $P_i$ in (\ref{projec}), it is straight forward to show that the stable endomorphism ring of 
$A_{0,j}$ is 
isomorphic to $\k$ for $j\in \{0,1,2\}$ and that $\Ext_{\Ar}^1(A_{0,j},A_{0,j})$ is isomorphic to $\k$ for $j\in \{0,2\}$ and zero for $j=1$. On the other hand, for $q\geq 1$ and for $j\in 
\{0,1,2\}$, the $\Ar$-module $A_{q,j}$ has a non-zero endomorphism which factors through $M[\1_i]$ and which does not factor through a projective $\Ar$-module. 
Assume that $r_i=2$. Since in this case $\mathfrak{A}_i$ is $\Omega$-stable, then for all $j\in \{0,1,2\}\mod 3$ and for all $q\geq 0$, the $\Ar$-module $B_{q,j}$ lies in the $\Omega$-
orbit of  $A_{q',j'}$ for some $j'\in \{0,1,2\}$ and $q'\geq 0$. In particular, $B_{0,1}=\Omega^{-1}A_{0,0}$, $B_{0,2}=\Omega^{-1}A_{0,1}$ and $B_{1,0}=\Omega^{-1}A_{0,2}$.  
If $r_i\geq 3$ then each of the modules $B_{0,1}$, $B_{0,2}$ and $B_{q,j}$ with $j\in \{0,1,2\}$ and $q\geq 1$ have a non-zero endomorphism factoring through $M[\1_i]$ and which 
does not factor through a projective $\Ar$-module.  Therefore, for all $r_i\geq 2$, the modules in $\mathfrak{A}_i\cup \Omega(\mathfrak{A}_i)$ whose stable endomorphism rings are 
isomorphic to $\k$ are precisely the modules in $\Omega$-orbits of the modules $A_{0,0}$, $A_{0,1}$ and $A_{0,2}$. 
 
Since $\Ext_{\Ar}^1(A_{0,1},A_{0,1})=0$, it follows that $R(\Ar,A_{0,1})\cong\k$. Since $\Ext_{\Ar}^1(A_{0,j},A_{0,j})$ is isomorphic to $\k$ for $j\in \{0,2\}$, it follows that $R(\Ar,A_
{0,j})$ is a quotient of $\k[[t]]$ for $j\in \{0,2\}$.

Let consider the $\Ar$-module $A_{0,0}=M[\1_i]$.
\begin{claimproof}\label{claim1} 
The universal deformation ring $R(\Ar,A_{0,0})$ of $A_{0,0}$ is isomorphic to $\k[[t]]/(t^{r_i})$.
\end{claimproof}

\begin{proofclaim}
For all $l\in\{0,\ldots,r_i-1\}$ let $S_l=\zeta_i^{-l}$. Then for all $l\in\{1,\ldots,r_i-1\}$ there exists a non-trivial canonical endomorphism $\sigma_l$ of the $\Ar$-module $M
[S_l]$ which factors through $M[S_{l-1}]$, namely
\begin{equation}\label{sigma1}
\sigma_l: M[S_l]\surjection M[S_{l-1}]\injection M[S_l]. 
\end{equation}
Observe that  the kernel of $\sigma_l$ and the image of $\sigma_l^{l-1}$ are isomorphic to $A_{0,0}$, and that $\sigma_l^l=0$. Thus, for all $l\in\{0,\ldots,r_i-1\}$, the $\Ar$-module 
$M[S_l]$ is naturally a $k[[t]]/(t^{l+1})\otimes_\k\Ar$-module where the action of $t$ over $m\in M[S_l]$ is given as $t\cdot m=\sigma_l(m)$. In particular, $tM[S_l]\cong M[S_{l-1}]$ for 
all $l\in\{1,\ldots,r_i-1\}$. 

Let $l\in\{1,\ldots,r_i-1\}$ be fixed and let  $\{\bar{b}_1\}$ be a $\k$-basis of $A_{0,0}$. Using the 
isomorphism $M[S_l]/tM[S_l]\cong A_{0,0}$, we can lift $\bar{b}_1$ to an element $b_1\in M[S_l]$. It follows that $\{b_1\}$ is linearly independent over $\k$ and that $\{t^ab_1: 0\leq 
a\leq l\}$ is a $\k$-basis of $tM[S_l]\cong M[S_{l-1}]$. Therefore, $\{b_1\}$ is a $\k[[t]]/(t^{l+1})$-basis of $M[S_l]$, which means that $M[S_l]$ is free over $\k[[t]]/(t^{l+1})$. Moreover, 
$M[S_l]$ lies in a short exact sequences of $\Ar$-modules 
\begin{equation*}
0\to tM[S_l]\to M[S_l]\to \k\otimes_{\k[[t]]/(t^{l+1})}M[S_l]\to 0.
\end{equation*}
Consequently, there exists an isomorphism of $\Ar$-modules $\phi_l:\k\otimes_{\k[[t]]/(t^{l+1})}M[S_l]\to A_{0,0}$, which implies that $(M[S_l],\phi_l)$ is a lift of $A_
{0,0}$ over $k[[t]]/(t^{l+1})$. Consider the lift $(M[S_{r_i-1}],\phi_{r_i-1})$ of $A_{0,0}$ over $\k[[t]]/(t^{r_i})$. Since $\End_{\Ar}(A_{0,0})\cong \k$ then by Theorem \ref{thm3}(i), there 
exists a unique morphism $\alpha:R(\Ar,A_{0,0})\to \k[[t]]/(t^{r_i})$ in 
$\hat{\Ca}$ such that $M[S_{r_i-1}]\cong \k[[t]]/(t^{r_i})\otimes_{R(\Ar,A_{0,0}),\alpha} U(\Ar,A_{0,0})$,  where $R(\Ar, A_{0,0})$ and $U(\Ar,A_{0,0})$ are respectively the universal 
deformation ring and the universal deformation of the $\Ar$-module $A_{0,0}$. Since $(M[S_1],\phi_1)$ is not the trivial lift of $A_{0,0}$ over $\k[[t]]/(t^2)$, it follows that there exists 
a unique surjective morphism $\alpha': R(\Ar,A_{0,0})\to \k[[t]]/(t^2)$ in $\hat{\Ca}$ such that $M[S_1]\cong k[[t]]/(t^2)\otimes_{R(\Ar,A_{0,0}),\alpha'}U(\Ar,A_{0,0})$. By considering the 
natural projection $\pi_{r_i,2}: \k[[t]]/(t^{r_i})\to \k[[t]]/(t^2)$ and the lift $(U',\phi_{U'})$  of $A_{0,0}$ over $\k[[t]]/(t^2)$ corresponding to the morphism $\pi_{r_i,2}\circ \alpha$, we obtain 
\begin{align*}
U'&\cong \k[[t]]/(t^2)\otimes_{R(\Ar,A_{0,0}),\pi_{r_i,2}\circ \alpha}U(\Ar,A_{0,0})\\
&\cong \k[[t]]/(t^2)\otimes_{\k[[t]]/(t^{r_i}),\pi_{r_i,2}}\left(\k[[t]]/(t^{r_i})\otimes_{R(\Ar,A_{0,0}),\alpha}U(\Ar,A_{0,0})\right)\\
&\cong \k[[t]]/(t^2)\otimes_{\k[[t]]/(t^{r_i}),\pi_{r_i,2}}M[S_{r_i-1}]\\
&\cong M[S_{r_i-1}]/t^2M[S_{r_i-1}]\cong M[S_1]. 
\end{align*} 
It follows from Remark \ref{rem1} that $[U',\phi_{U'}]=[M[S_1],\phi_1]$ in $\hat{\Fun}_{A_{0,0}}(\k[[t]]/(t^2))$. The uniqueness of $\alpha'$ implies $\alpha'= \pi_{r_i,2}\circ \alpha$. 
Since $\alpha'$ is surjective, it follows that $\alpha$ is also surjective.  We want to prove that $\alpha$ is an isomorphism. Suppose this is false.  Then there exists a surjective $\k$-
algebra 
homomorphism $\alpha_0: R(\Ar,A_{0,0})\to \k[[t]]/(t^{r_i+1})$ in $\hat{\Ca}$ such that $\pi_{r_i+1,r_i}\circ \alpha_0=\alpha$, where $\pi_{r_i+1,r_i}:\k[[t]]/(t^{r_i+1})\to \k[[t]]/(t^{r_i})$ is 
the natural projection. Let $M_0$ be a $\k[[t]]/(t^{r_i+1})\otimes_\k\Ar$-module which defines a lift of $A_{0,0}$ over $\k[[t]]/(t^{r_i+1})$ corresponding to $\alpha_0$. Since the kernel of 
$\pi_{r_i+1,r_i}$ is $(t^{r_i})/(t^{r_i+1})$, then $M_0/t^{r_i}M_0\cong M[S_{r_i-1}]$. Consider the $\k[[t]]/(t^{r_i+1})\otimes_\k\Ar$-module homomorphism $g:M_0\to t^{r_i}M_0$ defined by 
$g(m)=t^{r_i}m$ for all $m\in M_0$. Since $M_0$ is free over $\k[[t]]/(t^{r_i+1})$, if follows that the kernel of $g$ is isomorphic to $tM_0$. Since $g$ is a surjection, it follows that $M_0/
tM_0\cong t^{r_i}M_0$, which implies that $t^{r_i}M_0\cong A_{0,0}$. Hence, there exists a non-split short exact sequence of $\k[[t]]/(t^{r_i+1})\otimes_\k \Ar$-modules 
\begin{equation}\label{seq1}
0\to A_{0,0}\to M_0\to M[S_{r_i-1}]\to 0.
\end{equation}
Since $\Omega M[S_{r_i-1}] = \Omega M[\zeta_i^{-r_i+1}]=M[\underline{c}_i^{k-1}\tau_{i+2}\tau_{i+1}]$, then 
\[\Ext_{\Ar}^1(M[S_{r_i-1}],A_{0,0})=\underline{\Hom}_{\Ar}(\Omega M[S_{r_i-1}],A_{0,0})=0.\] 
It follows that the sequence (\ref{seq1}) splits as a sequence of $\Ar$-modules. Hence $M_0=A_{0,0}\oplus M[S_{r_i-1}]$ as $\Ar$-modules. Identifying the elements of $M_0$ as $
(a,m)$ with $a\in A_{0,0}$ and $m\in M[S_{r_i-1}]$ we see that the $t$ acts on $(a,m)\in M_0$ as $t\cdot(a,m)=(\mu(m),\sigma_{r_i-1}(m))$, where $\mu: M[S_{r_i-1}]\to A_{0,0}$ is a 
surjective $\Ar$-module homomorphism and $\sigma_{r_i-1}$ is as in (\ref{sigma1}). Since the canonical homomorphism $\epsilon: M[S_{r_i-1}] \surjection M[\1_i]\to M[\1_i]$ generates $\Hom_{\Ar}(M[S_{r_i-1}],A_{0,0})$, 
then there exists $c\in \k^\ast$ such that $\mu=c\epsilon$, which implies that the kernel of $\mu$ is $tM[S_{r_i-1}]$. Therefore $t^{r_i}(a,m)=(\mu(t^{r_i-1}m),\sigma_{r_i-1}^{r_i}
(m))=(0,0)$ for all $a\in A_{0,0}$ and $m\in M[S_{r_i-1}]$, which contradicts the fact that $t^{r_i}M_0\cong A_{0,0}$. Thus $\alpha: R(\Ar,A_{0,0}) \to \k[[t]]/(t^{r_i})$ is an isomorphism 
and $R(\Ar,A_{0,0})\cong \k[[t]]/(t^{r_i})$.  This finishes the proof of Claim \ref{claim1}.
\end{proofclaim}
Next consider the string $\Ar$-module $A_{0,2}=M[(\1_i)_{hh}]$.
\begin{claimproof}\label{claim2}
The universal deformation ring $R(\Ar,A_{0,2})$ of $A_{0,2}$ is isomorphic to $\k[[t]]$.
\end{claimproof}
\begin{proofclaim}
Let $T_0=(\1_i)_{hh}$ and for all $l\geq 1$, let $T_l=T_{l-1}\tau_i(\1_i)_{hh}$. Thus, for all $l\geq 1$ and by using similar arguments as those 
in the proof of Claim \ref{claim1}, we get lifts $(M[T_l], \varphi_l)$ of $A_{0,2}$ over $\k[[t]]/(t^{l+1})$, where for each $l\geq 1$,  $t$ acts on $m\in M[T_l]$ as $t\cdot 
m =\delta_l(m)$, where $
\delta_l$ is the non-trivial canonical endomorphism of $M[T_l]$ that factors through $M[T_{l-1}]$, namely 
\begin{equation}\label{sigma2}
\delta_l: M[T_l]\surjection M[T_{l-1}]\injection M[T_l].
\end{equation}
Note that for all $l\geq 1$, we have natural projections $\pi_{l,l-1}: M[T_l]\to M[T_{l-1}]$. Let $N_0=\invlim M[T_l]$ and let $t$ act on $N_0$ as $\invlim \pi_{l,l-1}$. In particular, $\k
\otimes_{\k[[t]]}N_0\cong N_0/tN_0\cong A_{0,2}$, which implies that there exists an isomorphism of $\Ar$-modules $\varphi_0:\k\otimes_{\k[[t]]}N_0\to A_{0,2}$. Let $n=\dim_\k A_
{0,2}$ and let $\{\bar{B}_j\}_{1\leq j\leq n}$ be a $\k$-basis of $N_0/tN_0$. For all $1\leq j\leq n$ , we are able to lift these elements $\bar{B}_j$ in $N_0/tN_0$ to elements $B_j$ of 
$N_0$ such that $\{B_j\}_{1\leq j\leq n}$ is a generating set of the $\k[[t]]\otimes_\k\Ar$-module $N_0$. It follows that $\{B_j\}_{1\leq j\leq n}$ is a $\k[[t]]$-basis 
of $N_0$, which implies that $N_0$ is free over $\k[[t]]$. Therefore, $(N_0,\varphi_0)$ is a lift of $A_{0,2}$ over $\k[[t]]$ and there exists a unique $\k$-algebra homomorphism $\beta:R(\Ar,A_{0,2})\to \k[[t]]$ in $\hat{\Ca}$ corresponding to the deformation defined by $(N_0,\varphi_0)$, where $R(\Ar,A_{0,2})$ is  the universal deformation ring of $A_{0,2}$. 
Since $N_0/t^2N_0\cong M[T_1]$ as $\Ar$-modules, we can see as in the proof of Claim \ref{claim1} that since $N_0/t^2N_0$ defines a non-trivial lift of $A_{0,2}$ over $\k[[t]]/
(t^2)$, then  $\beta$ is a surjection. Since $R(\Ar,A_{0,2})$ is a quotient of $\k[[t]]$, it follows that $\beta$ is an isomorphism. Hence $R(\Ar,A_{0,2})\cong \k[[t]]$. This finishes the proof 
of Claim \ref{claim2}, which finishes the proof of Proposition \ref{prop4}.
\end{proofclaim}
\end{proof}

\begin{proposition}\label{prop5}
For $i\in \{0,1,2\}\mod 3$, let $\mathfrak{B}_i$ be the component of $\Gamma_s(\Ar)$ containing the $\Ar$-module $M[\tau_i]$, where $
\bar{r}=(r_0,r_1,r_2,k)$ and $r_0,r_1,r_2\geq 3$, $k\geq 1$. Define 
\begin{align*}
(\tau_i)_h=\tau_i\underline{a}_{i+2}&&\text{ and }&&{_h}(\tau_i)=\underline{b}_{i+1}\tau_i.
\end{align*}
If $k=1$ then $\mathfrak{B}_i=\Omega(\mathfrak{A}_{i+2})$, where $\mathfrak{A}_{i+2}$ is as in Proposition \ref{prop4}. Thus, $\mathfrak{B}_i=\Omega(\mathfrak{B}_i)$ if and only if 
$k=1$ and $r_{i+2}=2$. 
The modules in $\mathfrak{B}_i\cup \Omega (\mathfrak{B}_i) $ whose stable endomorphism rings are isomorphic to $\k$ are precisely the modules in the $
\Omega$-orbits  of the modules $V_0=M[\tau_i]$, $V_1=M[(\tau_i)_h]$ and $V_{-1}=M[{_h}(\tau_i)]$. If $k=1$ then the universal deformation rings are 
\begin{align*}
R(\Ar,V_0)\cong \k,&&  R(\Ar,V_1)\cong \k[[t]]/(t^{r_{i+2}}), &&R(\Ar,V_{-1})\cong \k[[t]]. 
\end{align*}
 If $k\geq 2$ then the universal deformation rings are 
\begin{align*}
R(\Ar,V_0)\cong \k,&&  R(\Ar,V_1)\cong \k[[t]], &&R(\Ar,V_{-1})\cong \k[[t]]. 
\end{align*}

\end{proposition}   
\begin{proof}
Let $i\in\{0,1,2\}\mod 3$ be fixed. Using hooks and co-hooks (see \S \ref{ape2}) we see that all $\Ar$-modules in $\mathfrak{B}_i$ lie in the $\Omega$-orbit of either
\begin{align*}
C_{q,0}&=M[\tau_i(\underline{a}_{i+2}\underline{a}_{i+1}\underline{a}_i)^q], \text{ or }\\
C_{q,1}&=M[\tau_i(\underline{a}_{i+2}\underline{a}_{i+1}\underline{a}_i)^q\underline{a}_{i+2}], \text{ or }\\
C_{q,2}&=M[\tau_i(\underline{a}_{i+2}\underline{a}_{i+1}\underline{a}_i)^q\underline{a}_{i+2}\underline{a}_{i+1}], \text{ or }\\
D_{q,0}&=M[(\underline{b}_{i+2}\underline{b}_i\underline{b}_{i+1})^q\tau_i], \text { or }\\
D_{q,1}&=M[\underline{b}_{i+1}(\underline{b}_{i+2}\underline{b}_i\underline{b}_{i+1})^q\tau_i], \text { or }\\
D_{q,2}&=M[\underline{b}_i\underline{b}_{i+1}(\underline{b}_{i+2}\underline{b}_i\underline{b}_{i+1})^q\tau_i]
\end{align*}
for some $q\geq 0$. 

Note that $C_{0,0}=M[\tau_i]=D_{0,0}$, $C_{0,1}=M[(\tau_i)_h]$, $C_{0,2}=M[(\tau_i)_{hh}]$, $D_{0,1}=M[{_h(\tau_i)}]$ and $D_{0,2}=M[{_{hh}(\tau_i)}]$. 
By Proposition \ref{prop4}, it follows that if $k=1$ then $M[\tau_i]$ lies in $\Omega(\mathfrak{A}_{i+2})$, which implies that $\mathfrak{B}_i$ is $\Omega$-stable if and only if $r_{i+2}
=2$ and $k=1$. 

Using \S \ref{ape3} 
and the description of the projective indecomposable $\Ar$-module $P_i$ in (\ref{projec}), it is straight forward to show that the stable endomorphism rings of $C_{0,j}$ and $D_{0,j}
$ are isomorphic to $\k$ for $j\in \{0,1\}$, that $\Ext_{\Ar}^1(C_{0,1},C_{0,1})$ and $\Ext_{\Ar}^1(D_{0,1},D_{0,1})$ are isomorphic to $\k$,  and that $\Ext_{\Ar}^1(C_{0,0},C_{0,0})=0$.  
Moreover, $D_{0,2}$ and $D_{q,j}$ with $q\geq 1$ and $j\in\{0,1,2\}$ have a non-zero canonical endomorphism factoring through $M[\tau_i]$ that does not factor through a projective 
$\Ar$-module. If $k\geq 2$ or $r_{i+2}\geq 3$ then $C_{0,2}$ and $C_{q,j}$ with $q\geq 1$ and $j\in\{0,1,2\}$ have a non-zero canonical endomorphism which factors through 
$M[\1_{i+1}]$ and which does not factor through a projective $\Ar$-module. If $k=1$ and $r_{i+2}=2$ then $C_{0,2}=\Omega^{-1}C_{0,1}$, $C_{1,0}=\Omega^{-3}C_{0,0}$, $C_{1,1}=
\Omega^{-3}D_{0,1}$, and the modules $C_{1,2}$ and $C_{q,j}$ with $q\geq 2$, $j\in\{0,1,2\}$ have a non-trivial canonical endomorphism factoring through $M[\1_{i+1}]$ that does 
not factor through a projective $\Ar$-module.  Therefore, for all $r_0,r_1,r_2\geq 2$ and $k\geq 1$, the modules in $\mathfrak{B}_i\cup \Omega (\mathfrak{B}_i) $ whose stable 
endomorphism rings are isomorphic to $\k$ are precisely the modules in the $
\Omega$-orbits  of $C_{0,0}$, $C_{0,1}$ and $D_{0,1}$.
 
Since $\Ext_{\Ar}^1(C_{0,0},C_{0,0})=0$, it follows that $R(\Ar,C_{0,0})\cong\k$. Since $\Ext_{\Ar}^1(C_{0,1},C_{0,1})$ and $\Ext_{\Ar}^1(D_{0,1},D_{0,1})$ are both isomorphic to $
\k$ then $R(\Ar,C_{0,1})$ and $R(\Ar,D_{0,1})$ are quotients of $\k[[t]]$. Assume that $k=1$. Then by Theorem \ref{thm3} and Proposition \ref{prop4}, it follows that  
\[R(\Ar,C_{0,1})\cong R(\Ar,\Omega^{-2}M[\1_{i+2}])\cong \k[[t]]/(t^{r_{i+2}}),\] and 
\[R(\Ar,D_{0,1})\cong R(\Ar,\Omega M[(\1_{i+2})_{hh}])\cong \k[[t]].\] 

Next assume that $k\geq 2$. Let $S_0=(\tau_i)_h$, $T_0={_h(\tau_i)}$ and for all $l\geq 1$, let $S_l=S_{l-1}\tau_{i+1}(\tau_i)_h$ and $T_l=T_{l-1}\zeta_i^{-1} {_h(\tau_i)}$. Then by using similar 
arguments as in proof of Claim \ref{claim2} within the proof of Proposition \ref{prop4}, we obtain that $R(\Ar,C_{0,1})\cong \k[[t]]
\cong R(\Ar,D_{0,1})$. This finishes the proof of Proposition \ref{prop5}.
\end{proof}

Let $\mathfrak{C}_i$ be the component of $\Gamma_s(\Ar)$ containing the string module $M[\tau_{i+1}\tau_i]$ for some $i\in \{0,1,2\}\mod 3$. Observe that if $k=1$ then $
\mathfrak{C}_i$ is one of the $3$-tubes, otherwise $\mathfrak{C}_i$ is a component of type $\mathbb{ZA}_\infty^\infty$.  In Proposition \ref{prop7}, we determine the 
universal deformation rings of modules whose stable endomorphism ring is isomorphic to $\k$ lying in the $3$-tubes (see Proposition \ref{prop7}). In the following result, we assume 
that $k\geq 2$.

\begin{proposition}\label{prop6}
For $i\in \{0,1,2\}\mod 3$, let $\mathfrak{C}_i$ be the component of $\Gamma_s(\Ar)$ containing the $\Ar$-module $M[\tau_{i+1}\tau_i]$, 
where $\bar{r}=(r_0,r_1,r_2,k)$ and $r_0,r_1,r_2\geq 2$, $k\geq 2$. Define 
\begin{align*}
{_h}(\tau_{i+1}\tau_i)=\underline{b}_{i+2}\tau_{i+1}\tau_i&& \text{ and } && {_{hh}}(\tau_{i+1}\tau_i)=\underline{b}_{i+1}\underline{b}_{i+2}\tau_{i+1}\tau_i.
\end{align*} 
The component $\mathfrak{C}_i$ is $\Omega$-stable if and only if $k=2$. The modules in $\mathfrak{C}_i$ whose stable endomorphism ring is isomorphic to $\k$ are precisely the 
modules in the $\Omega$-orbits of the modules $W_0=M[\tau_{i+1}\tau_i]$, $W_{-1}=M[{_h}(\tau_{i+1}\tau_i)]$ and $W_{-2}=M[{_{hh}}(\tau_{i
+1}\tau_i)]$. Their universal deformation rings are 
\begin{align*}
R(\Ar,W_0)\cong \k[[t]]/(t^k),&& R(\Ar,W_{-1})\cong \k,&&R(\Ar,W_{-2})\cong \k[[t]]. 
\end{align*}
\end{proposition}  
\begin{proof}
Let $i\in\{0,1,2\}\mod 3$ be fixed. Using hooks and co-hooks (see \S \ref{ape2}) we see that all $\Ar$-modules in $\mathfrak{C}_i$ lie in the $\Omega$-orbit of either
\begin{align*}
E_{q,0}&=M[\tau_{i+1}\tau_i(\underline{a}_{i+2}\underline{a}_{i+1}\underline{a}_i)^q], \text{ or }\\
E_{q,1}&=M[\tau_{i+1}\tau_i(\underline{a}_{i+2}\underline{a}_{i+1}\underline{a}_i)^q\underline{a}_{i+2}], \text{ or }\\
E_{q,2}&=M[\tau_{i+1}\tau_i(\underline{a}_{i+2}\underline{a}_{i+1}\underline{a}_i)^q\underline{a}_{i+2}\underline{a}_{i+1}], \text{ or }\\
F_{q,0}&=M[(\underline{b}_i\underline{b}_{i+1}\underline{b}_{i+2})^q\tau_{i+1}\tau_i], \text { or }\\
F_{q,1}&=M[\underline{b}_{i+2}(\underline{b}_i\underline{b}_{i+1}\underline{b}_{i+2})^q\tau_{i+1}\tau_i], \text { or }\\
F_{q,2}&=M[\underline{b}_{i+1}\underline{b}_{i+2}(\underline{b}_i\underline{b}_{i+1}\underline{b}_{i+2})^q\tau_{i+1}\tau_i]
\end{align*}
for some $q\geq 0$. Note that $E_{0,0}=M[\tau_{i+1}\tau_i]=F_{0,0}$, $E_{0,1}=M[(\tau_{i+1}\tau_i)_h]$, $E_{0,2}=M[(\tau_{i+1}\tau_i)_{hh}]$,  $F_{0,1}=M[{_h(\tau_{i+1}\tau_i)}]$ 
and $F_{0,2}=M[{_{hh}(\tau_{i+2}\tau_i)}]$. Since $\Omega F_{0,0} =M[{_c}(\tau_{i+1}\tau_i\underline{c}_i^{k-2})]$, then $\mathfrak{C}_i$ is $\Omega$-stable if and only if $k=2$. 

By using \S \ref{ape3} and the description of the projective indecomposable $\Ar$-module $P_i$ in (\ref{projec}), it is straight 
forward to show that for all $j\in\{0,1,2\}$, the stable endomorphism ring of $F_{0,j}$ is isomorphic to $\k$ and for $q\geq 1$, the module $F_{q,j}$ has a non-trivial canonical 
endomorphism which factors through $M[\tau_{i+1}\tau_i]$ and which does not factor through a projective $\Ar$-module. 

Assume first that $k=2$. Since in this case $\mathfrak{C}_i$ is $\Omega$-stable, then for all $j\in \{0,1,2\}\mod 3$ and for all $q\geq 0$ the $\Ar$-module $E_{q,j}$ lies in the $\Omega
$-orbit of  $F_{q',j'}$ for some $j'\in \{0,1,2\}$ and $q'\geq 0$. In particular, $E_{0,1}= \Omega^{-1} F_{0,0}$, $E_{0,2}=\Omega^{-1} F_{0,1}$ and $E_{1,0}=\Omega^{-1} F_{0,2}$.  
Next assume that $k\geq 3$. Then for all $q\geq 0$ and $j\in\{0,1,2\}$ the module $E_{q,j}$ has a non-trivial canonical endomorphism, which factors through $M[\1_{i+2}]$ and which 
does not factor through a projective $\Ar$-module. Therefore for all $k\geq 2$,  the modules in $\mathfrak{C}_i$ whose stable endomorphism ring is isomorphic to $\k$ are precisely 
the modules in the $\Omega$-orbits  of the modules $F_{0,0}$, $F_{0,1}$ and $F_{0,2}$.

Since $\Ext_{\Ar}^1(F_{0,1},F_{0,1})=0$, it follows that $R(\Ar,F_{0,1})\cong \k$. Since $\Ext_{\Ar}^1(F_{0,0},F_{0,0})$ and $\Ext_{\Ar}^1(F_{0,2},F_{0,2})$ 
are isomorphic to $\k$ then $R(\Ar,F_{0,0})$ and $R(\Ar,F_{0,2})$ are quotients of $\k[[t]]$.  Let $T_0={_{hh}(\tau_{i+1}\tau_i)}$ and for all $0\leq j\leq k-1$ and $l\geq 1$, let $S_j=\underline{c}_{i+2}^j\tau_{i+1}\tau_i$ and $T_l={_{hh}}(\tau_{i+1}\tau_i)\zeta_i^{-1}T_{l-1}$. By using similar arguments as those in the proof of Proposition \ref{prop4}, we obtain that $R(\Ar,F_{0,0})\cong \k[[t]]/(t^k)$ and $R(\Ar,F_{0,2})\cong \k[[t]]$. This finishes the proof of Proposition \ref{prop6}.
\end{proof}

\subsection{$3$-tubes}
\begin{proposition}\label{prop7}
Let $\mathfrak{T_1}$ and $\mathfrak{T}_2$ be the two $3$-tubes of $\Gamma_s(\Ar)$, with $\bar{r}=(r_0,r_1,r_2,k)$ and $r_0,r_1,r_2\geq 2$ and $k\geq 
1$. Then $\Omega(\mathfrak{T}_1)=\mathfrak{T}_2$. Let  $T=\zeta_0^{-r_0+1}$ and define
\begin{align*}
 T_h=\zeta_0^{-r_0+1}\tau_2\zeta_2^{-r_2+1}&& \text{ and } && T_{hh}=\zeta_0^{-r_0+1}\tau_2\zeta_2^{-
r_2+1}\tau_1\zeta_1^{-r_1+1}.
\end{align*}
 The modules in $\mathfrak{T}_1\cup \mathfrak{T}_2$ whose stable endomorphism rings are isomorphic to $\k$ are precisely the modules in the $
\Omega$-orbit of $X_0=M[T]$, $X_1=M[T_h]$ and $X_2=M[T_{hh}]$. Their universal deformation rings are 
\begin{align*}
R(\Ar,X_0)\cong \k, && R(\Ar, X_1)\cong \k, &&  R(\Ar,X_2)\cong \k[[t]]
\end{align*}
\end{proposition}  
\begin{proof}
Using the description of the projective indecomposable $\Ar$-modules in (\ref{projec}), we see that $\Omega(\mathfrak{T}_1)=\mathfrak{T}_2$. 
Using \S \ref{ape3} and the description of the projective indecomposable $\Ar$-module $P_i$ in (\ref{projec}), it is straightforward to show that the only $\Ar$-modules in $
\mathfrak{T}_1\cup \mathfrak{T}_2$ whose stable endomorphism rings are isomorphic to $\k$ lie in the $\Omega$-orbit of either $X_0=M[T]$, $X_1=M[T_h]$ or $X_2=M[T_{hh}]$.  
Since $\Ext_{\Ar}^1(X_j, X_j)=0$ for $j\in \{0,1\}$, we have that $R(\Ar,X_j)\cong \k$ for $j\in \{0,1\}$. Since $\Ext_{\Ar}^1(X_2,X_2)$ is isomorphic to $\k$, it follows that $R
(\Ar,X_2)$ is a quotient of $\k[[t]]$. Let $S_0=T_{hh}$ and for all $l\geq 1$, let $S_l=S_{l-1}\tau_iT_{hh}$. By using similar arguments as those in the proof of Claim \ref{claim2} within the proof of Proposition \ref{prop4}, we obtain that $R(\Ar,X_2)\cong \k[[t]]$, which 
proves Proposition \ref{prop7}.  
\end{proof}


\bibliographystyle{amsplain}
\bibliography{bibliography}  

\providecommand{\bysame}{\leavevmode\hbox to3em{\hrulefill}\thinspace}
\providecommand{\MR}{\relax\ifhmode\unskip\space\fi MR }
\providecommand{\MRhref}[2]{%
  \href{http://www.ams.org/mathscinet-getitem?mr=#1}{#2}
}
\providecommand{\href}[2]{#2}
\begin{thebibliography}{10}

\bibitem{auslander}
M.~Auslander, I.~Reiten, and S.~Smal{\o}, \emph{Representation theory of
  {A}rtin algebras}, Cambridge Studies in Advanced Mathematics 36, Cambridge
  University Press, 1995.

\bibitem{benson}
D.~J. Benson, \emph{Representations and cohomology i: {B}asic representation
  theory of groups and associative algebras}, Cambridge Studies in Advanced
  Mathematics 30, Cambridge University Press, 1991.

\bibitem{bleher1}
F.~M. Bleher, \emph{Universal deformation rings of dihedral defect groups},
  Trans. Amer. Math. Soc \textbf{361} (2009), 3661--3705.

\bibitem{bleher2}
F.~M. Bleher and T.~Chinburg, \emph{Universal deformation rings and cyclic
  blocks}, Math. Ann. \textbf{318} (2000), 805--836.

\bibitem{bleher3}
\bysame, \emph{Universal deformation rings need not be complete intersections},
  Math. Ann. \textbf{337} (2007), 739--767.

\bibitem{bleher4}
F.~M. Bleher, T.~Chinburg, and B.~de~Smith, \emph{Inverse problems for
  deformation rings}, Trans. Amer. Math. Soc (2012), in press.

\bibitem{bleher5}
F.~M. Bleher and J.~B. Froelich, \emph{Universal deformation rings for the
  symmetric group ${S}_5$ and one of its double covers}, J. Pure Appl. Algebra
  \textbf{215} (2011), 523--530.

\bibitem{bleher6}
F.~M. Bleher and G.~Llosent, \emph{Universal deformation rings for the
  symmetric group ${S}_4$}, Algebr. Represent. Theory \textbf{13} (2010),
  255--270.

\bibitem{bleher7}
F.~M. Bleher, G.~Llosent, and J.~B. Schaefer, \emph{Universal deformation rings
  and dihedral blocks with two simple modules}, J. Algebra \textbf{345} (2011),
  49--71.

\bibitem{blehervelez}
F.~M. Bleher and J.~A. V{\'e}lez-Marulanda, \emph{Universal deformation rings
  of modules over {F}robenius algebras}, J. Algebra \textbf{367} (2012),
  176--202.

\bibitem{buri}
M.~C.~R. Butler and C.~M. Ringel, \emph{{A}uslander-{R}eiten sequences with few
  middle terms and applications to string algebras}, Comm. Algebra \textbf{15}
  (1987), 145--179.

\bibitem{curtis}
C.~W. Curtis and I.~Reiner, \emph{Methods of representation theory with
  applications to finite groups and orders}, vol.~I, John Wiley and Sons, New
  York, 1981.

\bibitem{erdmann}
K.~Erdmann, \emph{Blocks of tame representation type and related algebras},
  Lectures Notes in Mathematics 1428, Springer-Verlag, 1990.

\bibitem{ile}
R.~Ile, \emph{Change of rings in deformation theory of modules}, Trans. Amer.
  Math. Soc \textbf{356} (2004), 4873--4896.

\bibitem{krause}
H.~Krause, \emph{Maps between tree and band modules}, J. Algebra \textbf{137}
  (1991), 186--194.

\bibitem{mazur}
B.~Mazur, \emph{An introduction to the deformation theory of {G}alois
  representations}, Modular Forms and {F}ermat's Last Theorem, Springer-Verlag,
  Boston, MA, 1997, pp.~243--311.

\bibitem{sch}
M.~Schlessinger, \emph{Functors of {A}rtin rings}, Trans. Amer. Math. Soc.
  \textbf{130} (1968), 208--222.

\bibitem{yau}
D.~Yau, \emph{Deformation theory of modules}, Comm. Algebra \textbf{33} (2005),
  2351--2359.

\end{thebibliography}
\end{document}